\documentclass[12pt,a4paper]{article}
\usepackage[utf8]{inputenc}
\usepackage[T1]{fontenc}
\usepackage{amsmath}
\usepackage{amsfonts}
\usepackage{amssymb,amsthm}
\usepackage{graphicx}
\usepackage{float}
\usepackage{subcaption}
\usepackage[left=2.00cm, right=2.00cm, top=2.00cm, bottom=3.00cm]{geometry}
\usepackage{cite}

\numberwithin{equation}{section}

\newtheorem{theorem}{Theorem}[section]
\newtheorem{proposition}[theorem]{Proposition}
\newtheorem{corollary}[theorem]{Corollary}

\theoremstyle{definition}
\newtheorem{example}[theorem]{Example}
\newtheorem{definition}[theorem]{Definition}
\newtheorem{remark}[theorem]{Remark}
\allowdisplaybreaks

\title{Null surfaces of pseudo-spherical spacelike framed curves in the anti-de Sitter 3-space}
\author{\\\
        {\bf O. O\u{g}ulcan Tuncer}\thanks{E-mail address:
        otuncer@hacettepe.edu.tr}
 \medskip\\
   \\Department of Mathematics, Hacettepe University,
                    \\ 06800 Beytepe, Ankara, Turkey
}
\begin{document}

\maketitle
\begin{abstract}
We introduce null surfaces (or nullcone fronts) of pseudo-spherical spacelike framed curves in the three-dimensional anti-de Sitter space. These surfaces are formed by the light rays emitted from points on anti-de Sitter spacelike framed curves. We then classify singularities of the nullcone front of a pseudo-spherical spacelike framed curve and show how these singularities are related to the singularities of the associated framed curve. We also define a family of functions called the Anti-de Sitter distance-squared functions to explain the nullcone front of a pseudo-spherical spacelike framed curve as a wavefront from the viewpoint of the Legendrian singularity theory. We finally provide some examples to illustrate the results of this paper. \medskip
\end{abstract}
\section{Introduction}
Being among the simplest of curved spacetimes, $n$-dimensional Anti-de Sitter spacetime (AdS) has been of continuing interest not only to relativists but also to geometers. The local differential geometry of regular curves in the anti-de Sitter 3-space is still an active research area. But we cannot make use of standard tools to investigate curves with singularities. Many authors have studied certain families of these singular curves by introducing well-defined moving frames in different spaces, and these studies led to not only investigating several geometric properties of singular curves but also introducing important curves such as evolutes, involutes, pedals, and orthotomics associated to these singular curves 
\cite{chen,fukunaga,fukunaga2, fukunaga3, fukunaga4, honda, honda2,li,li2,lisun,li-tuncer,song, tuncer,tuncer2,yu}.  Recently we have introduced pseudo-spherical non-null framed curves in the anti-de Sitter 3-space and established the existence and uniqueness of these curves \cite{tuncerAdS}. After defining moving frames along these framed curves that are well-defined even at singular points, we have also investigated the geometric and singularity properties of evolutes and focal surfaces of the pseudo-spherical framed immersions.

Null submanifolds are one of the most appealing subjects from the viewpoint of singularity theory and general relativity. The null surfaces associated to a curve have been widely studied \cite{null1,null2,null3,null4,null5,null7,null8,null9}. These papers are mainly interested in investigating geometric properties and classifying singularities of these null surfaces. However, to the author's knowledge, the only paper dealing with the null surfaces associated to singular framed curves is \cite{null6}. 

The purpose of this paper is to investigate the singularities of null surfaces (or nullcone fronts) of pseudo-spherical spacelike framed curves in the anti-de Sitter 3-space. Therefore we generalize to null surfaces of singular spacelike framed curves the singularity result in \cite{null5} for null surfaces associated to regular spacelike curves in the anti-de Sitter 3-space.

This paper is organized as follows. In Section 2, a brief review of the local differential geometry of regular spacelike curves in the anti-de Sitter 3-space is given. In Section 3, pseudo-spherical spacelike framed curves in the anti-de Sitter 3-space are reviewed. In Section 4, the nullcone fronts of pseudo-spherical spacelike framed curves in the anti-de Sitter 3-space are introduced and the singularities of these surfaces are classified. In Section 5, using distance-squared functions on pseudo-spherical spacelike framed curves, the nullcone fronts are explained as wavefronts from the viewpoint of the Legendrian singularity theory. We finally provide some examples to support our results in Section 6.
\section{Preliminaries}
The four-dimensional real vector space with a pseudo-scalar product which is defined for $u=(u_1,u_2,u_3,u_4),\, w=(w_1,w_2,w_3,w_4)\in\mathbb{R}^4$ by 
\begin{equation*}
	\langle u,w\rangle=-u_1w_1-u_2w_2+u_3w_3+u_4w_4, 
\end{equation*}
is called the semi-Euclidean 4-space with index $2$ denoted by $\mathbb{R}^4_2$. \\ 
\indent A non-zero vector $u=(u_1,u_2,u_3,u_4)\in\mathbb{R}^4_2$ spacelike, timelike, or lightlike (null) if $\langle u,u\rangle>0$, $\langle u,u\rangle<0$, or $\langle u,u\rangle=0$, respectively. The pseudo-norm of the vector $u$ is defined by $\|u\|=\sqrt{\lvert\langle u,u\rangle\rvert}$. For three arbitrary vectors $u=(u_1,u_2,u_3,u_4)$, $v=(v_1,v_2,v_3,v_4)$, and $w=(w_1,w_2,w_3,w_4)$, the triple vector product is defined by
\begin{equation*}
	u\times v\times w = 
	\begin{vmatrix}
		-e_1& -e_2 & e_3 & e_4\\
		u_1& u_2 & u_3 &u_4\\
        v_1 & v_2 & v_3 & v_4\\
		w_1& w_2 & w_3 & w_4
	\end{vmatrix} 
\end{equation*}
where the set $\{e_1,e_2,e_3,e_4\}$ is the canonical basis of $\mathbb{R}^4_2$.\\ \indent
A curve in $\mathbb{R}^4_2$ is said to be spacelike, timelike, or lightlike (null) if the tangent vector of the curve is spacelike, timelike, or lightlike (null), respectively. 

There are three types of pseudo-spheres in $\mathbb{R}^4_2$; the anti-de Sitter 3-space, pseudo $3$-sphere with index 2, and nullcone at the origin respectively defined by
\begin{align*}
	&AdS^3=\{{u}\in\mathbb{R}^4_2\,\vert\,\langle {u},{u}\rangle=-1 \},\\
	&S^3_2=\{{u}\in\mathbb{R}^4_2\,\vert\,\langle {u},{u}\rangle=1 \},\\
	&\Lambda^3=\{{u}\in\mathbb{R}^4_2\backslash\{{0}\} \,\vert\,\langle {u},{u}\rangle=0 \}.
\end{align*} 
We now discuss the local differential geometry of regular spacelike curves in the anti-de Sitter 3-space. Let $\gamma:I\to AdS^3$ be a unit-speed spacelike curve. Let $T(s)=\gamma'(s)$ be the unit tangent vector. Since $\langle \gamma(s),\gamma(s) \rangle=-1$, we have $\langle \gamma(s), T(s)\rangle=0$. Then we find that $\langle \gamma(s), T'(s)\rangle =-1$. We take $N_1(s)=T'(s)-\gamma(s)$ and $N_2(s)=\gamma(s)\times T(s)\times N_1(s)$. It is easy to check that $N_1$ and $N_2$ are normal vectors of the spacelike curve $\gamma$ in $AdS^3$. These normal vectors can be spacelike or timelike. We also define the curvature by $\kappa_g(s)=\|T'(s)-\gamma(s)\|$. So we say that the spacelike curve $\gamma$ is a geodesic in $AdS^3$ if $\kappa_g(s)=0$ and $N_1(s)=0$. In the case of $\kappa_g(s)\neq0$, we are able to define the following unit vectors.
\[ n_1(s)=\dfrac{T'(s)-\gamma(s)}{\|T'(s)-\gamma(s)\|}=\dfrac{N_1(s)}{\|N_1(s)\|},\qquad n_2(s)=\gamma(s)\times T(s)\times n_1(s). \]
Hence the set $\{ \gamma(s), T(s), n_1(s), n_2(s)\}$ forms a pseudo-orthonormal frame along the spacelike curve $\gamma$. Then the Frenet-Serret type formulas are governed by
\begin{equation*}
	\begin{pmatrix}
		\gamma'(s)\\
		T'(s)\\
		n_1'(s)\\
        n_2'(s)
	\end{pmatrix}= \begin{pmatrix}
		0 & 1 & 0 & 0 
        \\ 1 & 0 & \kappa_g(s) & 0
        \\ 0 & -\delta \kappa_g(s) & 0 & \tau_g(s) 
        \\ 0 & 0 &  \tau_g(s) & 0
        
	\end{pmatrix}\begin{pmatrix}
		\gamma(s)\\
		T(s)\\
		n_1(s)\\
        n_2(s)
	\end{pmatrix},
\end{equation*}
where $\delta=\langle n_1(s), n_1(s)\rangle$ and $\tau_g(s)=\frac{\delta}{\kappa_g^2(s)}\det(\gamma(s),\gamma'(s),\gamma''(s),\gamma'''(s))$.
\section{Pseudo-spherical spacelike framed curves in the anti-de Sitter 3-space}
Let $\gamma:I\to AdS^3$ be a smooth curve. Then $(\gamma, v_1,v_2):I\to AdS^3\times \Delta_1$ is called a \textit{pseudo-spherical spacelike framed curve} if $(\gamma(s),v_1(s))^*\theta=0$ and $(\gamma(s),v_2(s))^*\theta=0$ for all $s\in I$, where
\begin{equation*}
	\Delta_1=\{(\mathbf{u},\mathbf{w})\,\vert\,\langle\mathbf{u},\mathbf{w}\rangle=0\}\subset AdS^3 \times {S}^3_2\,\,(\text{or}\,\,{S}^3_2\times AdS^3)
\end{equation*}
is a $4$-dimensional contact manifold, and $\theta$ is a canonical contact $1$-form on $\Delta_1$ \cite{CI}. The condition $(\gamma(s), v_i(s))^*\theta=0$ ($i=1,2$) is equivalent to $\langle \gamma'(s),v_i(s)\rangle=0$ ($i=1,2$) for all $s\in I$. If $(\gamma, v_1,v_2)$ is an immersion, then it is called a \textit{pseudo-spherical spacelike framed immersion}.

The curve $\gamma:I\to AdS^3$ is called a \textit{pseudo-spherical spacelike framed base curve} if there exists a smooth map $(v_1,v_2):I\to\Delta_1$ for which $(\gamma, v_1,v_2)$ is a pseudo-spherical spacelike framed curve.

Let $\mu(s)=\gamma(s)\times v_1(s)\times v_2(s)$. The set $\{\gamma(s), v_1(s),v_2(s),\mu(s) \}$ is a pseudo-orthonormal frame along $\gamma$. This frame is well-defined even at singular points of $\gamma$. The Frenet-Serret type formulas for this frame are given by
\begin{equation}\label{SpacelikeSF}
	\begin{pmatrix}
		\gamma'(s)\\
		v_1'(s)\\
		v_2'(s)\\
        \mu'(s)
	\end{pmatrix}= \begin{pmatrix}
		0 & 0 & 0 & \alpha(s) 
        \\ 0 & 0 & \ell(s) & m(s)
        \\ 0 & \ell(s) & 0 & n(s)
        \\ \alpha(s) & -\epsilon m(s) &  \epsilon n(s) & 0
        
	\end{pmatrix}\begin{pmatrix}
	\gamma(s)\\
		v_1(s)\\
		v_2(s)\\
        \mu(s)
	\end{pmatrix},
\end{equation}
where $\epsilon=\langle v_1(s),v_1(s)\rangle$, $\alpha(s)=\langle \gamma'(s), \mu(s)\rangle$, $\ell(s)=-\epsilon\langle v_1'(s), v_2(s)\rangle$, $m(s)=\langle v_1'(s), \mu(s)\rangle$, and $n(s)=\langle v_2'(s), \mu(s)\rangle$. We call the mapping $(\alpha, \ell, m,n):I\to\mathbb{R}^4$ the \textit{curvature} of the pseudo-spherical spacelike framed curve $(\gamma,v_1,v_2)$. Notice that $s_0$ is a singular point of $\gamma$ if and only if $\alpha(s_0)=0$. See \cite{tuncerAdS} for further details.

\begin{theorem}[Existence of pseudo-spherical spacelike framed curves \cite{tuncerAdS}]
For a smooth mapping $(\alpha,\ell,m,n):I\to\mathbb{R}^4$, there exists a pseudo-spherical spacelike framed curve $(\gamma,v_1,v_2)$ such that $\alpha$, $\ell$, $m$, and $n$ are the curvatures of $\gamma$.
\end{theorem}

\begin{theorem}[Uniqueness of pseudo-spherical spacelike framed curves \cite{tuncerAdS}]
Let $(\gamma, v_1, v_2)$ and $(\Tilde{\gamma}, \Tilde{v}_1, \Tilde{v}_2)$ be two pseudo-spherical spacelike framed curves in $AdS^3$. Suppose that the curvatures $(\alpha,\ell,m,n)$ and $(\Tilde{\alpha}, \Tilde{\ell},\Tilde{m}, \Tilde{n})$ of these two framed curves coincide. Then $(\gamma, v_1, v_2)$ and $(\Tilde{\gamma}, \Tilde{v}_1, \Tilde{v}_2)$ are congruent as pseudo-spherical spacelike framed curves.
\end{theorem}

\section{Nullcone fronts of pseudo-spherical spacelike framed curves}
In this section we define nullcone fronts of pseudo-spherical spacelike framed curves by using the pseudo-orthonormal frame $\{\gamma(s), v_1(s),v_2(s),\mu(s) \}$. We then classify the singularities of these surfaces.
\begin{definition}
   Let $(\gamma,v_1,v_2):I\to AdS^3\times \Delta_1$ be a pseudo-spherical spacelike framed curve. The \textit{nullcone front} of $\gamma$ is defined by 
   \begin{equation}\label{nullconefront}
       \mathcal{NF}^\pm_\gamma(s,\lambda)=\gamma(s)+\lambda(v_1(s)\pm v_2(s)).
   \end{equation}
\end{definition}
Notice that $v_1(s)\pm v_2(s)$ in \eqref{nullconefront} is a null vector. So these nullcone fronts can be considered as the surfaces formed by the trajectories of light rays emitted from each point on a spacelike framed curve in the anti-de Sitter 3-space.

Before dividing into the main theorem of this paper, we need to recall the criteria for singularities of fronts (See \cite{Izumiya-Circular,kokubu} for  details). A smooth map $f:U\subset\mathbb{R}^2\to AdS^3$ is a frontal if there exists a unit vector field $\nu:U\to T_1(AdS^3)$ along $f$ such that $\langle df(X), \nu \rangle(p)=0$ for any $X\in T_pU$. $f$ is called a front if $(f,\nu)$ is an immersion. Given a local coordinate system $(u,v)$ of $U$, there exists a smooth function $\Omega(u,v)$ on $U$ called the signed area density such that
\[ f(u,v)\times \dfrac{\partial}{\partial u}f(u,v)\times \dfrac{\partial}{\partial v}f(u,v)=\Omega(u,v)\nu(u,v). \]
Then we have $\Omega^{-1}(0)=\mathcal{S}(f)$, where $\mathcal{S}(f)$ is the set of singular points of $f$. A singular point $p$ is called non-degenerate if $d\Omega(p)\neq0$. For a non-degenerate point $p$, there exists a regular curve $c(s):I\to U$ such that $c(s_0)=p$ and $\text{image}(c)=\mathcal{S}(f)$ near $p$. Then we have a non-zero vector field $\xi(s)$ near $p$ called the null vector field along $c$ such that $df(\xi(s))=0$ for all $s$. 
\begin{theorem}\label{thmsingcri}
    Let $f:U\to AdS^3$ be a front and let $p=c(s_0)\in U$ be a non-degenerate singular point of $f$. 
    \begin{enumerate}
        \item[{\normalfont (i)}] The germ of $f$ at $p$ is locally diffeomorphic to the cuspidal edge if and only if $\xi(s_0)$ is transversal to $c'(s_0)$, i.e., $\det(c'(s_0),\xi(s_0))\neq 0$.
        \item[{\normalfont (ii)}] The germ of $f$ at $p$ is locally diffeomorphic to the swallowtail if and only if $\det(c'(s_0),\xi(s_0))= 0$ and $\dfrac{d}{ds}\det(c'(s),\xi(s))(s_0)\neq 0$.
    \end{enumerate}
    Here $C\times\mathbb{R}=\left\{ (x_1,x_2)\,|\,x_1^2-x_2^3=0 \right\}\times \mathbb{R}$ is the cuspidal edge and \\ $SW=\left\{ (x_1,x_2,x_3)\,|\, x_1=3u^4+u^2v,\,\, x_2=4u^3+2uv,\,\,x_3=v\right\}$ is the swallowtail (see Figure \ref{figcesw}).
\end{theorem}
\begin{figure}[H]
    \centering
    \begin{minipage}{.5\textwidth}
        \centering
        \includegraphics[width=0.8\textwidth]{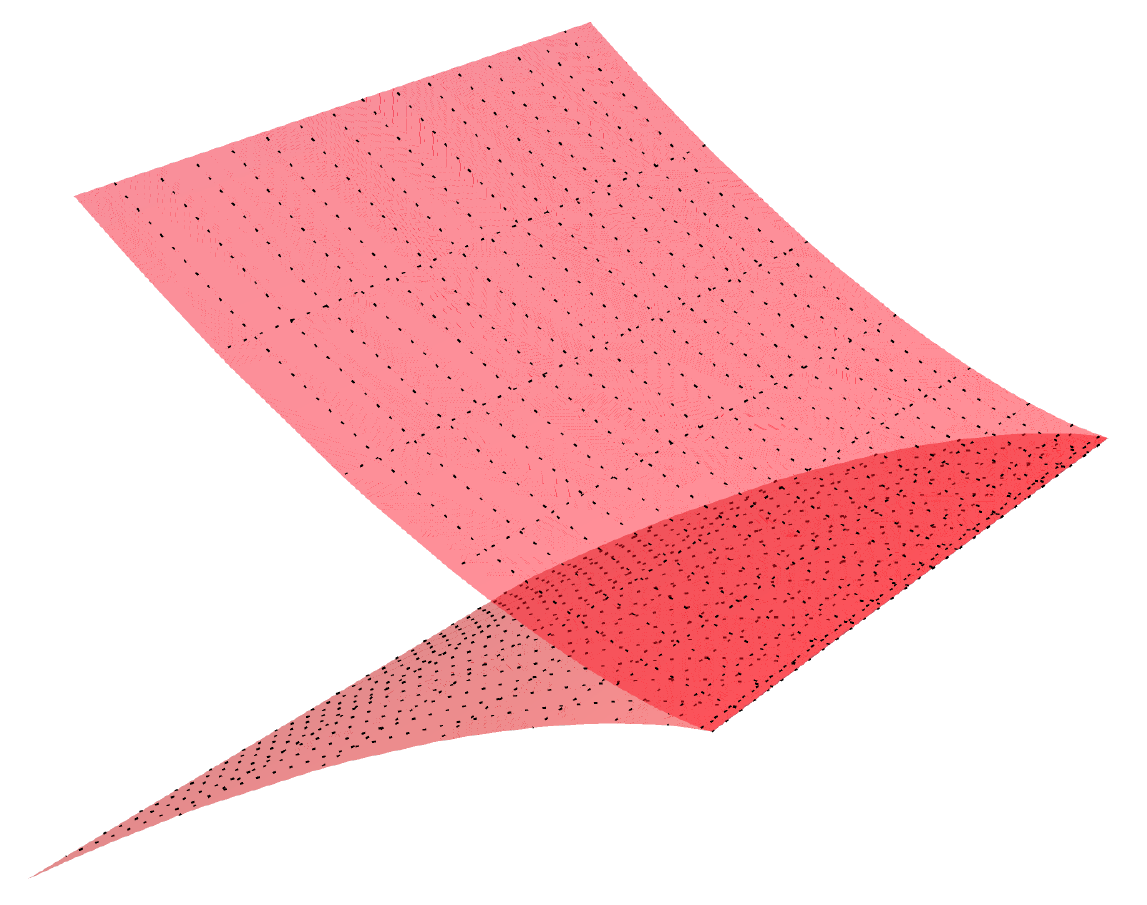}
    \end{minipage}%
    \begin{minipage}{0.5\textwidth}
        \centering
        \includegraphics[width=0.7\textwidth]{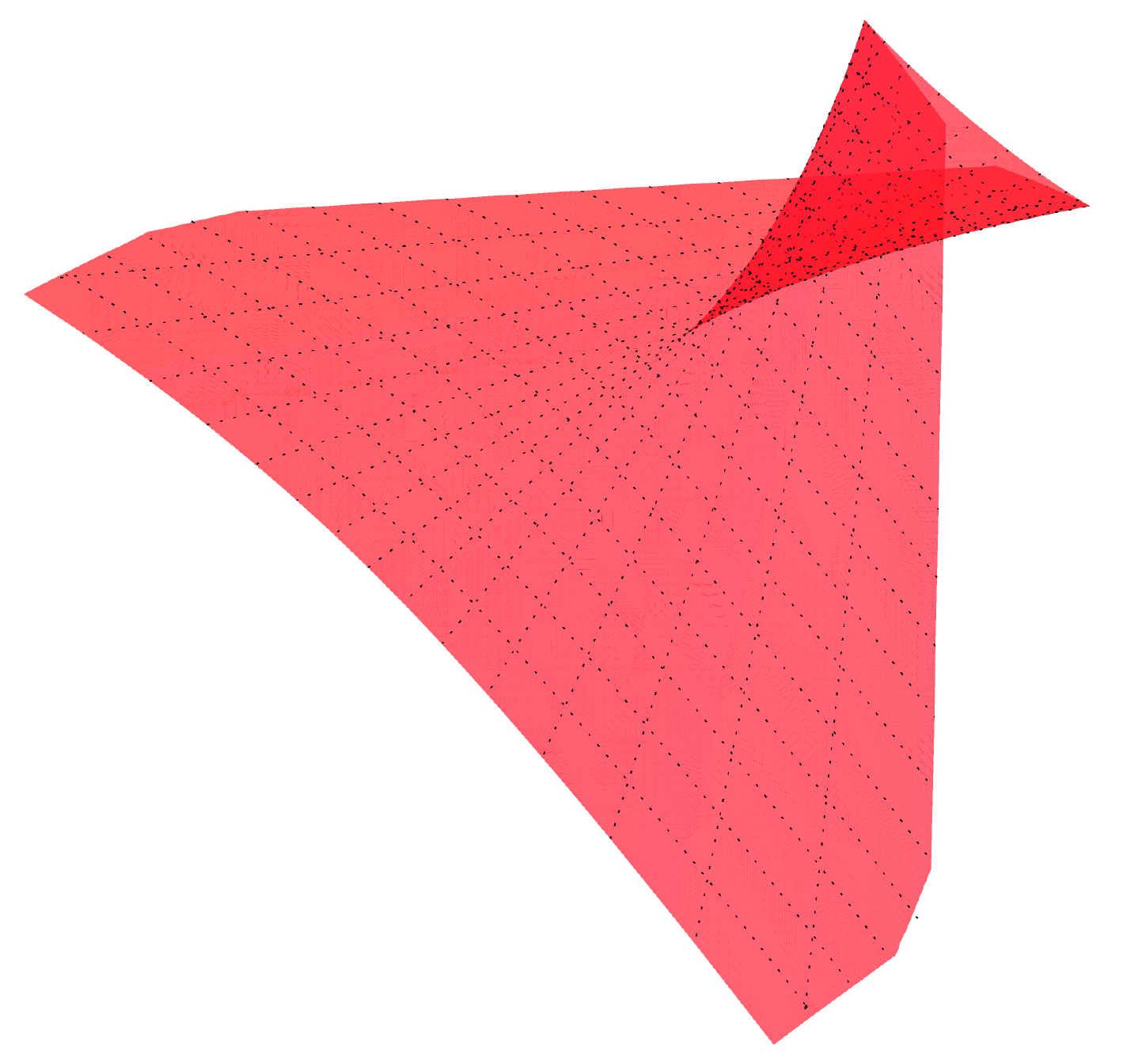}
    \end{minipage}
    \caption{Cuspidal edge (left) and Swallowtail (Right)    \label{figcesw}}
\end{figure}
Now we classify the singularities of the nullcone front defined in \eqref{nullconefront}. So our purpose in this paper is to prove the following theorem.
\begin{theorem}\label{mainthm}
    Let $(\gamma,v_1,v_2):I\to AdS^3\times \Delta_1$ be a pseudo-spherical spacelike framed curve with the curvature $({\alpha},{\ell},{m},{n})$ such that $m(s)\pm n(s)\neq 0$. Then we have the following.
    \begin{enumerate}
        \item[{\normalfont (i)}] The nullcone front $\mathcal{NF}^\pm_\gamma$ of $\gamma$ has singularity at $(s_0,\lambda_0)$ if and only if $\lambda_0=\dfrac{-\alpha(s_0)}{m(s_0)\pm n(s_0)}$.
         \item[{\normalfont (ii)}] The nullcone front $\mathcal{NF}^\pm_\gamma$ of $\gamma$ is locally diffeomorphic to the cuspidal edge at $(s_0,\lambda_0)$ if and only if $\lambda_0=\dfrac{-\alpha(s_0)}{m(s_0)\pm n(s_0)}$ and $\sigma(s_0)\neq 0$.
         \item[{\normalfont (iii)}] The nullcone front $\mathcal{NF}^\pm_\gamma$ of $\gamma$ is locally diffeomorphic to the swallowtail at $(s_0,\lambda_0)$ if and only if $\lambda_0=\dfrac{-\alpha(s_0)}{m(s_0)\pm n(s_0)}$, $\sigma(s_0)=0$, and $\sigma'(s_0)\neq 0$.
    \end{enumerate}
    Here 
    \[\sigma(s)=\alpha(s)\big( -(m'(s)\pm n'(s))+\ell(s)(n(s)\pm m(s))\big)+\alpha'(s)(m(s)\pm n(s)). \]
\end{theorem}
\begin{proof}
Let us calculate the partial derivatives of $\mathcal{NF}_\gamma^\pm(s,\lambda)$. Differentiating \eqref{nullconefront} with respect to $s$ and using \eqref{SpacelikeSF}, we find that
\begin{equation}
    \dfrac{\partial}{\partial s}\mathcal{NF}_\gamma^\pm(s,\lambda)=(\alpha(s)+\lambda(m(s)\pm n(s)))\mu(s)+\lambda\ell(s)(v_2(s)\pm v_1(s)). \label{ncone1}
\end{equation}
Differentiating \eqref{nullconefront} with respect to $\lambda$ gives
\begin{equation}
    \label{ncone2} 
    \dfrac{\partial}{\partial\theta}\mathcal{NF}_\gamma^\pm(s,\lambda)=v_1(s)\pm v_2(s).
\end{equation}
The nullcone surface of $\gamma$ defined by \eqref{nullconefront} has singularity at $(s_0,\theta_0)$ if and only if $\mathcal{NF}_\gamma^\pm\times \frac{\partial}{\partial s}\mathcal{NF}_\gamma^\pm\times \frac{\partial}{\partial \theta}\mathcal{NF}_\gamma^\pm(s_0,\lambda_0)=0$. Then from \eqref{nullconefront}, \eqref{ncone1}, and \eqref{ncone2}, this triple vector product is equal to 
\[ -\left(\alpha(s)+\lambda(m(s)\pm n(s))\right)(v_2(s)\pm v_1(s)), \]
which is zero at $(s_0,\lambda_0)$ if and only if 
\begin{equation}\label{ncone3}
   \alpha(s_0)+\lambda_0(m(s_0)\pm n(s_0))=0.
\end{equation}
So (i) immediately follows from the assumption $m(s)\pm n(s)\neq 0$. 

 We consider the signed density function 
\begin{equation*}
    \Omega(s,\lambda)=-(\alpha(s)+\lambda(m(s)\pm n(s))).
\end{equation*}
Set $\Omega^{-1}(0)=\mathcal{S}(\mathcal{NF}^\pm_\gamma)$. We see that $\mathcal{S}(\mathcal{NF}^\pm_\gamma)=\{(s,\lambda(s))\}$, where $\lambda(s)$ is a function satisfying $\Omega(s,\lambda(s))=0$. Then by assumption, we have 
\[ \dfrac{\partial}{\partial\lambda}\Omega(s,\lambda)=-(m(s)\pm n(s)) \neq 0. \]
Therefore any $p\in \mathcal{S}(\mathcal{NF}^\pm_\gamma)$ is non-degenerate. Let $p$ be a non-degenerate singular point. Then there exists a regular curve $c:I\to I\times\mathbb{R}\subset \mathbb{R}^2$ such that $c(0)=p$ and $\text{image}(c)=\mathcal{S}(\mathcal{NF}^\pm_\gamma)$ near $p$. Let $c(s)=(s,\lambda(s))$. Consider the null vector field $\xi:I\to\mathbb{R}^2$ along $c(s)$ given by $\xi(s)=(1,\pm\alpha(s)\ell(s)/(m(s)\pm n(s)))$. Then from \eqref{ncone3}
\begin{align*}
\det(c'(s_0), \xi(s_0))&=\begin{vmatrix}
    1 & -\dfrac{\alpha'(s_0)(m(s_0)\pm n(s_0))-\alpha(s_0)(m'(s_0)\pm n'(s_0))}{(m(s_0)\pm n(s_0))^2} \medskip\\
    1 & \pm \dfrac{\alpha(s_0)\ell(s_0)}{m(s_0)\pm n(s_0)}
\end{vmatrix} \\
&=\dfrac{\pm\alpha(s_0)\ell(s_0)(m(s_0)\pm n(s_0))+\alpha'(s_0)(m(s_0)\pm n(s_0))-\alpha(s_0)(m'(s_0)\pm n'(s_0))}{(m(s_0)\pm n(s_0))^2}\\
    &=\dfrac{\sigma(s_0)}{(m(s_0)\pm n(s_0))^2}.
\end{align*}
Thus by Theorem \ref{thmsingcri}(i), $\mathcal{NF}_\gamma^\pm$ is locally diffeomorphic to the cuspidal edge at $(s_0,\lambda_0)$ if and only if $\sigma(s_0)\neq 0$. This gives (ii).

It is easy to calculate that 
\[ \dfrac{d}{ds}=\det(c'(s), \xi(s))=\dfrac{\sigma'(s)(m(s)\pm n(s))^2-\sigma(s)(m'(s)\pm n'(s))}{(m(s)\pm n(s))^4}. \]
From Theorem \ref{thmsingcri}(ii),  $\mathcal{NF}_\gamma^\pm$ is locally diffeomorphic to the swallowtail at $(s_0,\lambda_0)$ if and only if $\sigma(s_0)=0$ and   $\sigma'(s_0)\neq 0$. Hence the proof of (iii) is completed.
\end{proof}
\begin{remark}
    Suppose that $s_0$ is a singular point of $\gamma(s)$. Then $(s_0,0)$ is a singular point of $\mathcal{NF}_{\gamma}^\pm$.
\end{remark}
The following corollary is a direct consequence of Theorem \ref{mainthm}. 
\begin{corollary}\label{cor1}
Let $(\gamma,v_1,v_2):I\to AdS^3\times \Delta_1$ be a pseudo-spherical framed curve with singularity at $s_0$, that is, $\alpha(s_0)=0$. Suppose that $m(s)\pm n(s)\neq0$ for all $s\in I$.
\begin{enumerate}
    \item[{\normalfont (i)}] The nullcone front $\mathcal{NF}_{\gamma}^\pm$ is locally diffeomorphic to the cuspidal edge at $(s_0,0)$ if and only if $\alpha'(s_0)\neq 0$.
    \item[{\normalfont (ii)}] The nullcone front $\mathcal{NF}_{\gamma}^\pm$ is locally diffeomorphic to the swallowtail at $(s_0,0)$ if and only if $\alpha'(s_0)=0$, $\alpha''(s_0)\neq0$.
\end{enumerate}
\end{corollary}

\section{Distance-squared functions on pseudo-spherical spacelike framed curves}
Given a pseudo-spherical spacelike framed curve $(\gamma,v_1,v_2):I\to AdS^3\times \Delta_1$, we now show how to explain the null surface of $\gamma$ as a wavefront from the viewpoint of Legendrian singularity theory. Define the following families of functions $D:I\times AdS^3\to\mathbb{R}$ by $D(s,\mathbf{v})=\langle \gamma(s)-\mathbf{v},\gamma(s)-\mathbf{v}\rangle$ called \textit{the Anti-de Sitter distance-squared function}. For any $\mathbf{v}_0\in AdS^3$, set $d_{\mathbf{v}_0}(s)=D(s,\mathbf{v}_0)$. The following proposition follows from a straightforward calculation.
\begin{proposition}
   Let $(\gamma,v_1,v_2):I\to AdS^3\times \Delta_1$ be a pseudo-spherical spacelike framed curve with the curvature $(\alpha, \ell, m, n)$. Suppose that $m(s)\pm n(s)\neq 0$ for all $s\in I$. Let $\sigma(s)=\alpha(s)\big( -(m'(s)\pm n'(s))+\ell(s)(n(s)\pm m(s))\big)+\alpha'(s)(m(s)\pm n(s))$.
   \begin{enumerate}
       \item[{\normalfont (1)}] $d_{\mathbf{v}_0}(s_0)=0$ if and only if there exist $a,b,c\in\mathbb{R}$ such that $\mathbf{v}_0=\gamma(s_0)+av_1(s_0)+bv_2(s_0)+c\mu(s_0)$, $c^2=-\epsilon(a^2-b^2)$.
       \item[{\normalfont (2)}] $d_{\mathbf{v}_0}(s_0)=d'_{\mathbf{v}_0}(s_0)=0$ if and only if $\alpha(s_0)$ or there exists a $\lambda\in\mathbb{R}$ such that  $\mathbf{v}=\gamma(s_0)+\lambda(v_1(s_0)\pm v_2(s_0))$.
\item[{\normalfont (3)}] $d_{\mathbf{v}_0}(s_0)=d'_{\mathbf{v}_0}(s_0)=d''_{\mathbf{v}_0}(s_0)=0$ if and only if at least one of the following conditions is satisfied
\begin{enumerate}
    \item[{\normalfont (i)}] $\alpha(s_0)=\alpha'(s_0)=0$,
    \item[{\normalfont (ii)}] $\alpha(s_0)$ and there exists a $\lambda\in\mathbb{R}$ such that $\mathbf{v}_0=\gamma(s_0)+\lambda(v_1(s_0)\pm v_2(s_0))$,
    \item[{\normalfont (iii)}] $\mathbf{v}_0=\gamma(s_0)-\dfrac{\alpha(s_0)}{m(s_0)\pm n(s_0)}(v_1(s_0)\pm v_2(s_0))$.
\end{enumerate}
\item[{\normalfont (4)}] $d_{\mathbf{v}_0}(s_0)=d'_{\mathbf{v}_0}(s_0)=d''_{\mathbf{v}_0}(s_0)=d'''_{\mathbf{v}_0}(s_0)=0$ if and only if at least one of the following conditions is satisfied
\begin{enumerate}
    \item[{\normalfont (i)}] $\alpha(s_0)=\alpha'(s_0)=\alpha''(s_0)=0$,
    \item[{\normalfont (ii)}] $\alpha(s_0)=\alpha'(s_0)=0$ and there exists a $\lambda\in\mathbb{R}$ such that $\mathbf{v}_0=\gamma(s_0)+\lambda(v_1(s_0)\pm v_2(s_0))$,
    \item[{\normalfont (iii)}] $\alpha(s_0)=0$ and $\mathbf{v}_0=\gamma(s_0)$,
    \item[{\normalfont (iv)}] $\mathbf{v}_0=\gamma(s_0)-\dfrac{\alpha(s_0)}{m(s_0)\pm n(s_0)}(v_1(s_0)\pm v_2(s_0))$ and $\sigma(s_0)=0$.
\end{enumerate}
\item[{\normalfont (5)}] $d_{\mathbf{v}_0}(s_0)=d'_{\mathbf{v}_0}(s_0)=d''_{\mathbf{v}_0}(s_0)=d'''_{\mathbf{v}_0}(s_0)=d^{(4)}_{\mathbf{v}_0}(s_0)=0$ if and only if at least one of the following conditions is satisfied
\begin{enumerate}
    \item[{\normalfont (i)}] $\alpha(s_0)=\alpha'(s_0)=\alpha''(s_0)=\alpha'''(s_0)=0$,
    \item[{\normalfont (ii)}] $\alpha(s_0)=\alpha'(s_0)=\alpha''(s_0)=0$ and there exists a $\lambda\in\mathbb{R}$ such that $\mathbf{v}_0=\gamma(s_0)+\lambda(v_1(s_0)\pm v_2(s_0))$,
    \item[{\normalfont (iii)}] $\alpha(s_0)=\alpha'(s_0)=0$ and $\mathbf{v}_0=\gamma(s_0)$,
    \item[{\normalfont (iv)}] $\mathbf{v}_0=\gamma(s_0)-\dfrac{\alpha(s_0)}{m(s_0)\pm n(s_0)}(v_1(s_0)\pm v_2(s_0))$, $\sigma(s_0)=0$, and $\sigma'(s_0)=0$.
\end{enumerate}
   \end{enumerate}
\end{proposition}
%
\section{Examples}
\begin{example}
Consider the smooth curve $\gamma_1:I\to AdS^3$ defined by
\[ \gamma_1(s)=\dfrac{1}{\sqrt{2}}\left(\sqrt{1+s^4}, \sqrt{1+s^6}, s^2, s^3\right). \]
The derivative of this curve with respect to $s$ is
\[ \gamma_1'(s)=\dfrac{1}{\sqrt{2}}\left( \dfrac{2s^3}{\sqrt{1+s^4}}, \dfrac{3 s^5}{\sqrt{1+s^6}},2s,3s^2 \right). \]
Therefore, the curve $\gamma$ is singular at $s=0$. Define $v_1:I\to S^3_2$ and $v_2:I\to AdS^3$ by
\begin{align*}
    v_1(s)&=\dfrac{1}{\sqrt{2(8+18s^2+s^6)}}(s^3\sqrt{1+s^4},s^3\sqrt{1+s^6},s^5+6s, s^6-4),\\
    v_2(s)&=\dfrac{1}{\sqrt{8+18s^2+s^6}\sqrt{4+9s^2+13s^6}}\big(-\sqrt{1+s^4}(4+9s^2-2s^6), \sqrt{1+s^6}(4+9s^2-2s^6),\\
    &\qquad\qquad\qquad\qquad\qquad\qquad\qquad\qquad 2s^2(-2+3s^2+s^6),3s^3(-2+3s^2+s^6)\big).
\end{align*}
It is easy to see that $\langle v_1, \gamma_1\rangle=0$, $\langle v_2, \gamma_1\rangle=0$, $\langle v_1, \gamma_1'\rangle=0$, and $\langle v_2, \gamma_1'\rangle=0$. Thus $(\gamma_1,v_1,v_2):I\to AdS^3\times \Delta_1$ is a pseudo-spherical spacelike framed curve in $AdS^3$. From the triple vector product $\gamma_1\times v_1\times v_2$, we find that
\[ \mu(s)=\dfrac{\sqrt{1+s^4}\sqrt{1+s^6}}{\sqrt{4+9s^2+13s^6}}\left(\dfrac{2s^2}{\sqrt{1+s^4}},\dfrac{3s^4}{\sqrt{1+s^6}},2,3s\right). \]
The curvature of $\gamma_1$ is given by $(\alpha,\ell,m,n)$, where
\begin{align*}
    \alpha(s)&=\dfrac{s\sqrt{4+9s^2+13s^6}}{\sqrt{2}\sqrt{1+s^4}\sqrt{1+s^6}},\\
    \ell(s)&=\dfrac{6\sqrt{2}s^2(2-3s^2-s^6)}{(8+18s^2+s^6)\sqrt{4+9s^2+13s^6}},\\
    m(s)&=\dfrac{12+16s^4+21s^6+25s^{10}}{\sqrt{2}\sqrt{1+s^4}\sqrt{1+s^6}\sqrt{8+18s^2+s^6}\sqrt{4+9s^2+13s^6}},\\
    n(s)&=\dfrac{s(-16+30s^2+81s^4+58s^6+102s^8+65s^{12})}{\sqrt{1+s^4}\sqrt{1+s^6}\sqrt{8+18s^2+s^6}(4+9s^2+13s^6)}.
\end{align*}
Then the nullcone front $\mathcal{NF}_{\gamma_1}^+$ is
\begin{align*}
    \mathcal{NF}_{\gamma_1}^+(s,\lambda)=&\bigg( \dfrac{\sqrt{1+s^4}}{2}\left( \sqrt{2}+\dfrac{\left(-8+s^2(-18+4s^4+s\sqrt{8+18s^2+26s^6}) \right)\lambda}{\sqrt{8+18s^2+s^6}\sqrt{4+9s^2+13s^6}} \right), \\
    & \dfrac{\sqrt{1+s^6}}{2}\left( \sqrt{2}+\dfrac{\left(8+s^2(18+6s^4+s\sqrt{8+18s^2+26s^6}) \right)\lambda}{\sqrt{8+18s^2+s^6}\sqrt{4+9s^2+13s^6}} \right), \\
    & \dfrac{s}{2}\left( \sqrt{2}s+\dfrac{\left(-8s+12s^3+4s^7+(6+s^4)\sqrt{8+18s^2+26s^6} \right)\lambda}{\sqrt{8+18s^2+s^6}\sqrt{4+9s^2+13s^6}} \right), \\
    & \dfrac{s^3}{\sqrt{2}}+\dfrac{\left(-12s^3+18s^5+6s^9+(s^6-4)\sqrt{8+18s^2+26s^6} \right)\lambda}{2\sqrt{8+18s^2+s^6}\sqrt{4+9s^2+13s^6}}  \bigg).
\end{align*}
The projections of $\gamma_1$ and $\mathcal{NF}_{\gamma_1}^+$ onto the $x_2x_3x_4$-space are visualized in Figure \ref{fig1}.
Notice that $\alpha(0)=0$ and $\alpha'(0)\neq 0$. Then by Corollary \ref{cor1}(i) $\mathcal{NF}_{\gamma_1}^+$ is locally diffeomorphic to the cuspidal edge at $(0,0)$.
\begin{figure}[H]
		\centering
		\includegraphics[width=0.6\textwidth]{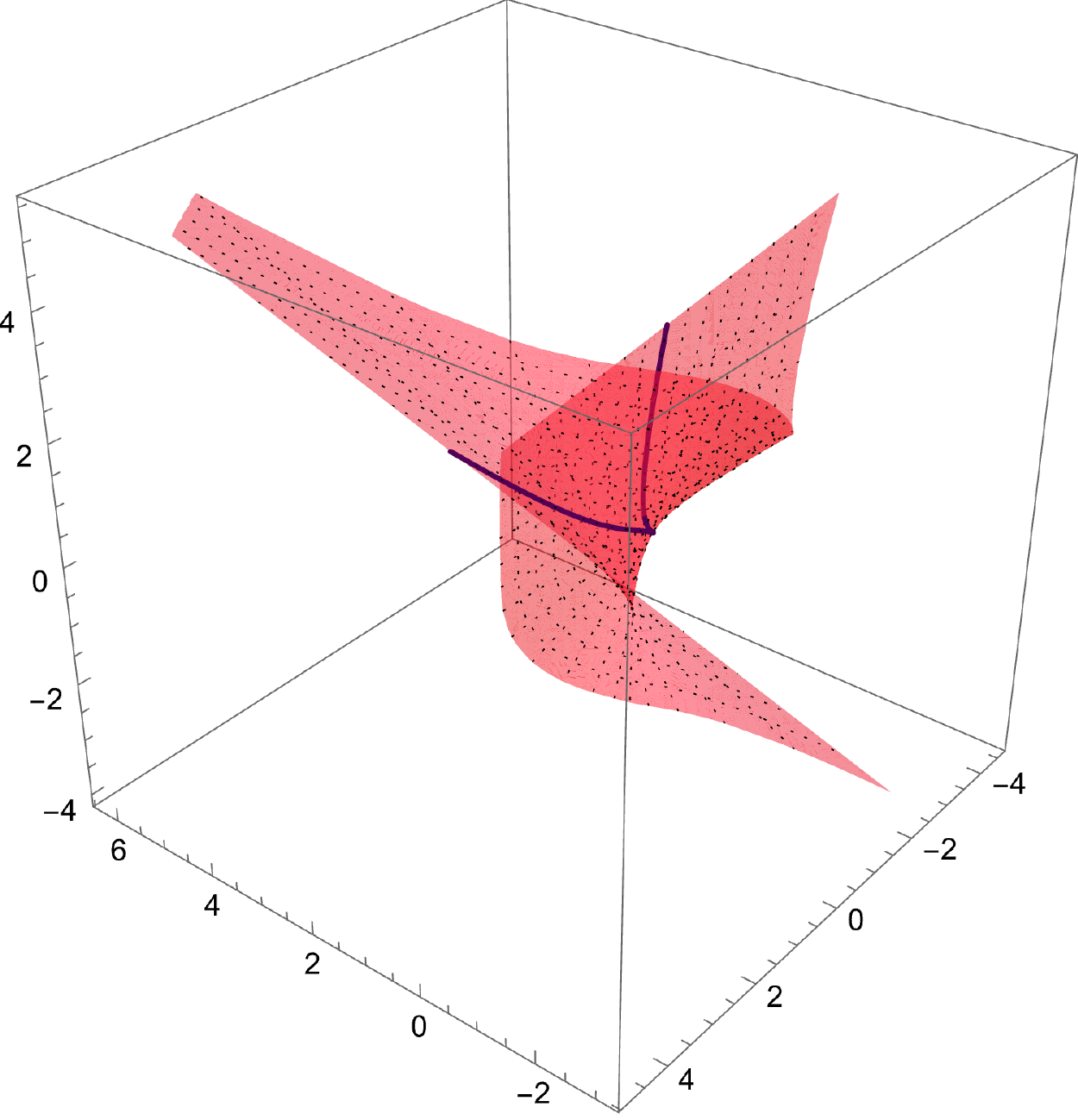}
		\caption{The projection of the nullcone front $\mathcal{NF}_{\gamma_1}^+$ onto the $x_2x_3x_4$-space} \label{fig1}
	\end{figure}
\end{example}
\begin{example}
Consider the curve $\gamma_2(s)=(0,\sqrt{1+\sin^6 s+\cos^6s},\cos^3(s), \sin^3(s))$. Letting 	
\[v_1(s)=\dfrac{1}{\sqrt{1+\sin^2s\cos^2s}}\bigg(0,\sin s\cos s\sqrt{1+\cos^6s+\sin^6s},\sin s(1+\cos^4s),\cos s(1+\sin^4s)\bigg), \]
and $v_2(s)=(1,0,0,0)$, we see that $(\gamma_2,v_1,v_2):I\to AdS^3\times\Delta_1$ is a pseudo-spherical spacelike framed curve. We also get
	\[\mu(s)=\dfrac{\sqrt{1+\cos^6s+\sin^6s}}{\sqrt{1+\sin^2s\cos^2s}}\bigg(0,\dfrac{\cos^4s-\sin^4 s}{\sqrt{1+\cos^6s+\sin^6s}},\cos s,-\sin s\bigg). \]
 The curvature of this framed curve is given by $(\alpha,\ell,m,n)$, where
 \[ \alpha(s)=-\dfrac{3\sin2s\sqrt{4+\sin^2(2s)}}{\sqrt{26+6\cos4s}},\quad \ell(s)=0,\quad m(s)=\dfrac{302+216\cos4s-6\cos8s}{8(9-\cos4s)\sqrt{26+6\cos4s}}, \quad n(s)=0. \]
 Then the nullcone front $\mathcal{NF}_{\gamma_2}^+$ is
\begin{align*}
    \mathcal{NF}_{\gamma_2}^+(s,\lambda)=&\bigg(\lambda,  \sqrt{1+\cos^6s+\sin^6s}\left(1+\dfrac{\lambda\cos s\,\sin s}{\sqrt{1+\cos^2s\sin^2s}} \right), \\
    & \: \cos^3s+\dfrac{\sin s\,(1+\cos^4 s)\lambda}{\sqrt{1+\cos^2s\sin^2s}},\, \sin^3s+\dfrac{\cos s\,(1+\sin^4 s)\lambda}{\sqrt{1+\cos^2s\sin^2s}}  \bigg).
\end{align*}
The projections of $\gamma_2$ and $\mathcal{NF}_{\gamma_2}^+$ onto the $x_2x_3x_4$-space are visualized in Figure \ref{fig2}. By Corollary \ref{cor1}(i), this nullcone front is locally diffeomorphic to the cuspidal edge at $(\rho,0)$ since $\alpha(\rho)=0$ and $\alpha'(\rho)\neq0$, where $\rho=0,\pi/2,\pi,2\pi$.
    \begin{figure}[H]
		\centering
		\includegraphics[width=0.6\textwidth]{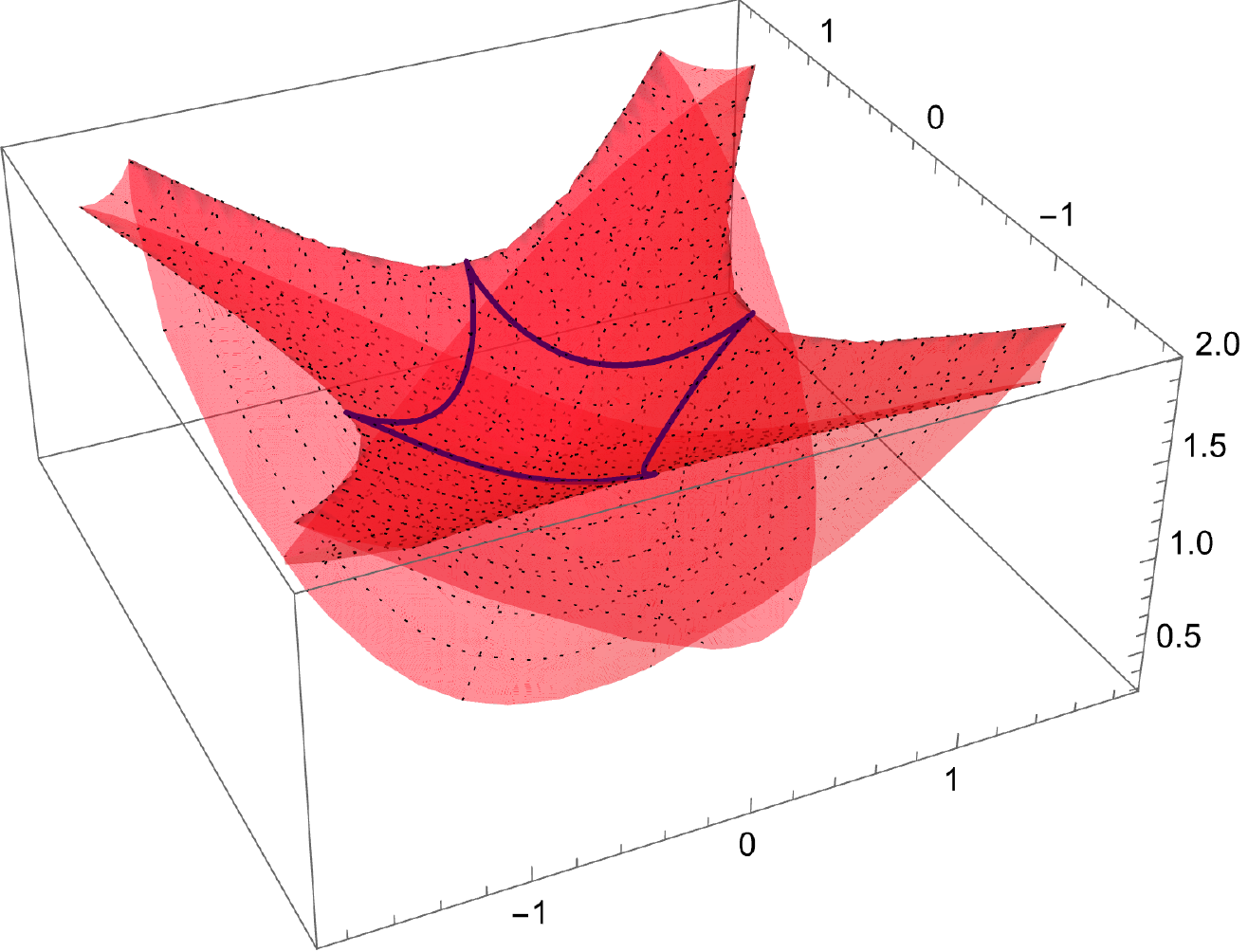}
		\caption{The projection of the nullcone front $\mathcal{NF}_{\gamma_2}^+$ onto the $x_2x_3x_4$-space} \label{fig2}
	\end{figure}
\end{example}
\begin{example}
	Consider the curve $\gamma_3(s)=(0,\sqrt{1+s^6+s^{8}}, s^3,s^4)$. Differentiating this equation with respect to $s$ yields
	\[ \gamma_3'(s)=(0,\dfrac{3s^5+4s^{7}}{\sqrt{1+s^6+s^{8}}},\,3s^2,\,4s^3).\]
	Taking 
	\[v_1(s)=\dfrac{1}{\sqrt{s^{8}+8s^{2}+9}}\big(0,s^4\sqrt{1+s^6+s^{8}},s^7+4s,s^{8}-3\big), \]
 and 
 \[ v_2(s)=(1,0,0,0), \]
	we find that $(\gamma_3,v_1,v_2):I\to AdS^3\times\Delta_1$ is a pseudo-spherical spacelike framed curve with the curvature $(\alpha,\ell,m,n)$, where
	\[\alpha(s)=\dfrac{s^2\sqrt{s^{8}+8s^{2}+9}}{\sqrt{1+s^6+s^{8}}},\quad \ell(s)=0, \quad m(s)=\dfrac{s^{14}+28s^{8}+21s^6+12}{(s^{8}+8s^{2}+9)\sqrt{1+s^6+s^{8}}}, \quad n(s)=0.\]
Then the nullcone front $\mathcal{NF}_{\gamma_3}^+$ is
\begin{align*}
    \mathcal{NF}_{\gamma_3}^+(s,\lambda)=&\bigg(\lambda,  \sqrt{1+s^6+s^{8}}\left(1+\dfrac{s^4\lambda}{\sqrt{9+8s^2+s^{8}}} \right), \\
    & \: s^3+\dfrac{s(4+s^6)\lambda}{\sqrt{9+8s^2+s^{8}}},\, s^4+\dfrac{(-3+s^{8})\lambda}{\sqrt{9+8s^2+s^{8}}}  \bigg).
\end{align*}
The projections of $\gamma_3$ and $\mathcal{NF}_{\gamma_3}^+$ onto the $x_2x_3x_4$-space are visualized in Figure \ref{fig3}. Notice that $\alpha(0)=0$, $\alpha'(0)=0$, and $\alpha''(0)\neq 0$. Then by Corollary \ref{cor1}(ii) $\mathcal{NF}_{\gamma_3}^+$ is locally diffeomorphic to the swallowtail at $(0,0)$.
	\begin{figure}[H]
		\centering
		\includegraphics[width=0.65\textwidth]{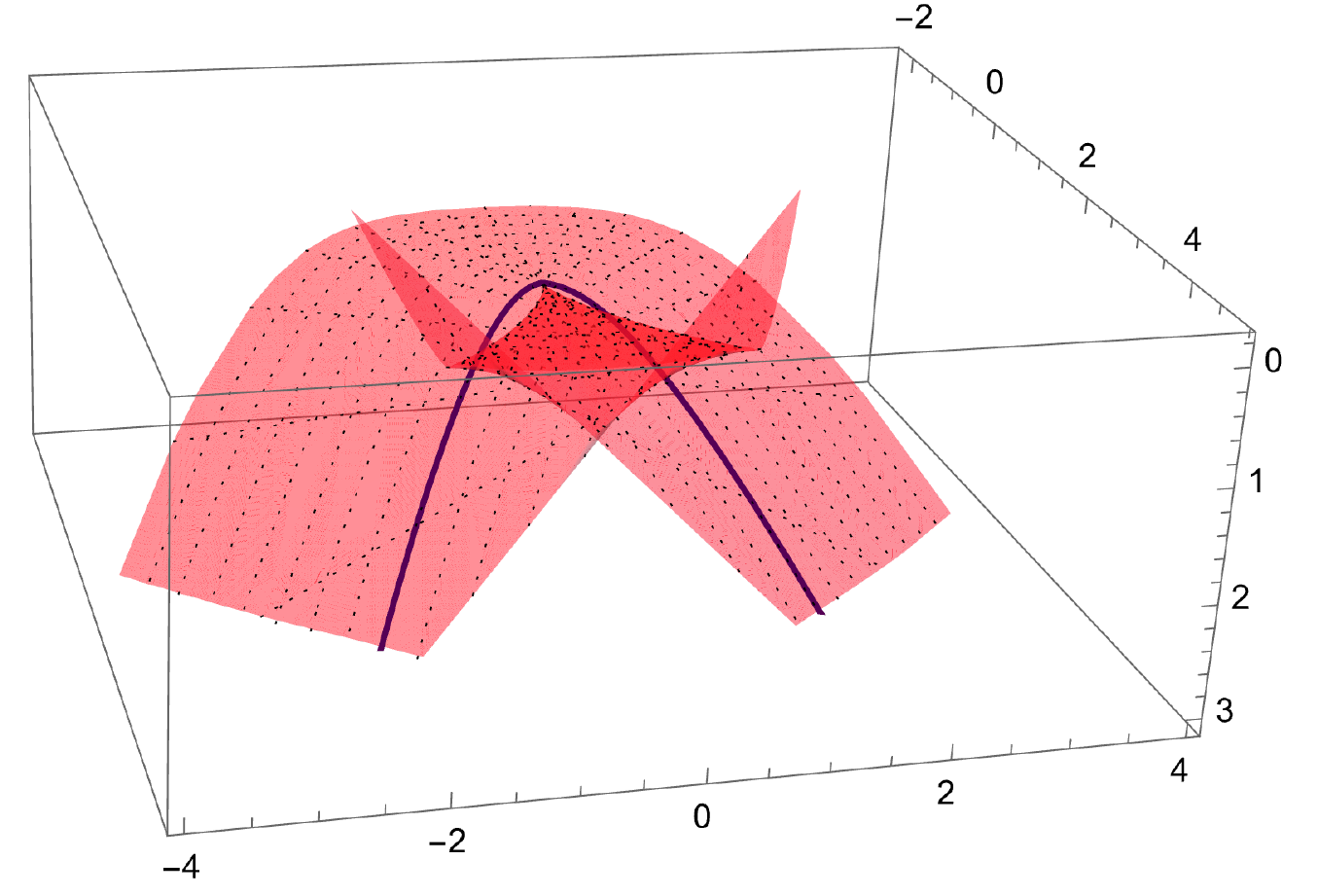}
		\caption{The projection of the nullcone front $\mathcal{NF}_{\gamma_3}^+$ onto the $x_2x_3x_4$-space} \label{fig3}
	\end{figure}
\end{example}


\end{document}